\documentclass[12pt,psamsfonts]{amsart}
\usepackage{euscript}
\usepackage{amsmath}
\usepackage{amsthm}
\usepackage{amssymb}
\usepackage{amscd}
\usepackage{amsfonts}
\usepackage{amsbsy}
\usepackage{mathrsfs}
\usepackage{epsfig}
\usepackage{setspace}
\usepackage{graphicx}
\usepackage{xcolor}

\newtheorem {theorem} {Theorem}

\newcommand{\R}{\ensuremath{\mathbb{R}}}
\newcommand{\aS}{\ensuremath{\mathbb{S}}}
\newcommand{\D}{\ensuremath{\mathbb{D}}}
\newcommand{\X}{\mathcal{X}}
\newcommand{\Y}{\mathcal{Y}}
\newcommand{\p}{\partial }
\newcommand{\la}{\lambda}

\begin{document}

\title[Integrability and dynamics of the Basener-Ross population model]
{Darboux integrability and dynamics\\ of the Basener-Ross population model}

\author[F. G\"ung\"or, J. Llibre  and C. Pantazi]
{Faruk G\"ung\"or$^1$, Jaume Llibre$^2$ and Chara Pantazi$^3$}

\address{$^1$ Department of Mathematics, Faculty of Science and Letters, Istanbul Technical University, 34469 Istanbul, Turkey}
\email{gungorf@itu.edu.tr}

\address{$^2$ Departament de Matem\`{a}tiques, Universitat Aut\`{o}noma de Barcelona, 08193 Bellaterra, Barcelona, Catalonia, Spain}
\email{jllibre@mat.uab.cat}

\address{$^3$ Departament de Matem\`atica Aplicada I, Universitat Polit\`ecnica de Cata\-lunya, (EPSEB), Av. Doctor Mara\~{n}\'on, 44--50, 08028 Barcelona, Spain}
\email{chara.pantazi@upc.edu}


\date{}
\dedicatory{}

\maketitle

\begin{abstract}
We deal with the Basener and Ross model for the evolution of human population in Easter island. We study the Darboux integrability of this model and  characterize all its global dynamics in the Poincar\'e disc, obtaining $15$ different topological phase portraits.
\end{abstract}

\section{Introduction and statement of the main results}

In order to explore the evolution of ecosystems an isolated island is a good laboratory due to the total absence of external distortion factors like migration. Basener and Ross, see equation (5) in \cite{BR}, proposed the following model for the evolution of human population in Easter island
\begin{equation}\label{e1}
\dot{x}=x(1-y),  \quad \dot{y}=(h-1)y^2+(1-c)y+\frac{c}{k}x,
\end{equation}
where $c, k, h$ are positive constants, $c$ is the growth rate of the recourses, $k$ the carrying capacity of the population, $h$ is the harvesting constant, $y(t)$ is the quotient between the amount of resources $x(t)$ and the human population in the island at time $t$.

It is known that the population of Easter Island grew regularly for some time and then diminished very rapidly; to the extent that humans almost disappeared from the island. Basener and Ross \cite{BR} provided the mathematical model \eqref{e1} that allowed to explain this type of behavior. With respect to other predator-prey models the Basener--Ross model shows a rich variety of dynamical behaviors, allowing the extinction in finite time. Therefore it has been considered as an acceptable model for the evolution of population in ancient civilizations and some generalizations have been done and studied, for instance see \cite{BB,BF,BM,Ko} and the references therein. On the other hand, the Basener--Ross model has been extended considering that its ecological parameters can change with the time, see for instance \cite{Am,FF,GT,HB,SR,ZH}.

Nucci and Sanchini \cite{NS} applied the Lie group theory to system \eqref{e1} and proved that this system can be integrated by quadrature for some values of the parameters. They also provided a comparison analysis with the qualitative study given by Basener and Ross.

We have two objectives, first to study the Darboux first integrals of system \eqref{e1}, and second to characterize all the phase portraits of the differential system \eqref{e1}, thus completing the initial qualitative analysis done by Basener and Ross. More precisely, we will classify all the phase portraits of the differential system \eqref{e1} in the Poincar\'e disc. Thus, in particular we control all the orbits which come or go to infinity, which never were studied previously for system \eqref{e1}.

By scaling the time and the $x$ coordinate we can reduce the study to the values of the parameters $k=c$, and the analysis of the differential system \eqref{e1} is reduced to study the differential system
\begin{equation*}
\dot{x}=x(1-y),  \quad \dot{y}=(h-1)y^2+(1-c)y+x.
\end{equation*}
Additionally, setting $b=h-1$ we have
\begin{equation}\label{e2}
\dot{x}=x(1-y)=P(x,y),  \quad \dot{y}=by^2+(1-c)y+x=Q(x,y),
\end{equation}
with $b>-1$ and $c>0$.

Darboux [38] showed how can be constructed the first integrals of planar polynomial vector fields possessing sufficient invariant algebraic curves.  System \eqref{e2} is {\it integrable} on an open subset $U$ of $\R^2$ if there exists a nonconstant analytic function $H :U \rightarrow \R$, called a {\it first integral} of the system on $U$, which is constant on all solution curves $(x(t), y(t))$ of system \eqref{e2} contained in $U$. We say that an analytic function $H(x,y,t): U\times \R \rightarrow \R$ is an {\it invariant} of system \eqref{e2} on $U$, if $H(x,y,t)= \mbox{constant}$ for all values of $t$ for which the solution $(x(t), y(t))$ is defined and contained in $U$. If an invariant $H$ is independent of $t$ then, of course, it is a first integral.

The knowledge provided by an invariant is weaker than the one provided by a first integral. The invariant, in general, only gives
information about either the $\alpha$-- or the $\omega$--limit set of the orbits of the system (see for instance \cite{LO}), while the level curves of a first integral contain the orbits of the system.

Let $f=f(x,y)$ be a real polynomial in the variables $x$ and $y$. The algebraic curve $f(x,y)=0$ is an {\it invariant algebraic curve} of system \eqref{e2} if for some polynomial $K=K(x,y)$ we have
\begin{equation}\label{e3}
P\frac{\p f}{\p x}+ Q\frac{\p f}{\p y}= Kf.
\end{equation}
The polynomial $K$ is called the {\it cofactor} of the invariant algebraic curve $f=0$. We note that since the polynomial system has degree $m$, then any cofactor has at most degree $m-1$. Since on the points of the algebraic curve $f=0$ the gradient $(\p f/\p x$, $\p f/\p y)$ of the curve is orthogonal to the vector field $(P,Q)$ associated to system \eqref{e2}, the vector field $(P,Q)$ is tangent to the curve $f=0$ at every point of this curve. Hence, the curve $f=0$ is formed by orbits of system \eqref{e2}. This justifies the name of invariant algebraic curve given to the algebraic curve $f=0$ satisfying (\ref{e3}) for some polynomial $K$, because it is {\it 	invariant} under the flow defined by system \eqref{e2}.

In the next theorem we summarize the basic results on the Darboux theory of integrability that we shall use in this paper, for a proof see \cite{Da} or Chapter 8 of \cite{DLA}.

\begin{theorem}\label{t0}
Suppose that a polynomial system
\begin{equation}\label{e4}
\dot x=P(x,y), \qquad \dot y=Q(x,y),
\end{equation}
admits $p$ irreducible invariant algebraic curves $f_i=0$ with cofactors $K_i$ for $i=1,\ldots,p$.
\begin{itemize}
\item[(a)] There exist $\la_i$'s in $\R$ not all zero such that $\mathop{\sum}\limits_{i=1}^p\la_i K_i =0$, if and only if the function
\begin{equation}\label{e5}
f_1^{\la_1}\cdots f_p^{\la_p}
\end{equation}
is a first integral of system \eqref{e4}.
	
\item[(b)] There exist $\la_i$'s in $\R$ not all zero such that $\mathop{\sum}\limits_{i=1}^p\la_i K_i =-s$ for some $s\in \R\setminus \{0\}$, if and only if the function 	
\begin{equation}\label{e6}
f_1^{\la_1}\cdots f_p^{\la_p} e^{st}
\end{equation}
is an invariant of system \eqref{e4}.
\end{itemize}
\end{theorem}

The first integrals of the form \eqref{e5} are called the {\it Darboux first integrals}, and the invariants of the form \eqref{e6} are called the {\it Darboux invariants}.

Our main result on the Darboux integrability of the Basener--Ross differential system \eqref{e2} is the following.

\begin{figure}[ht]
$$
\begin{array}{ccc}
\epsfig{file=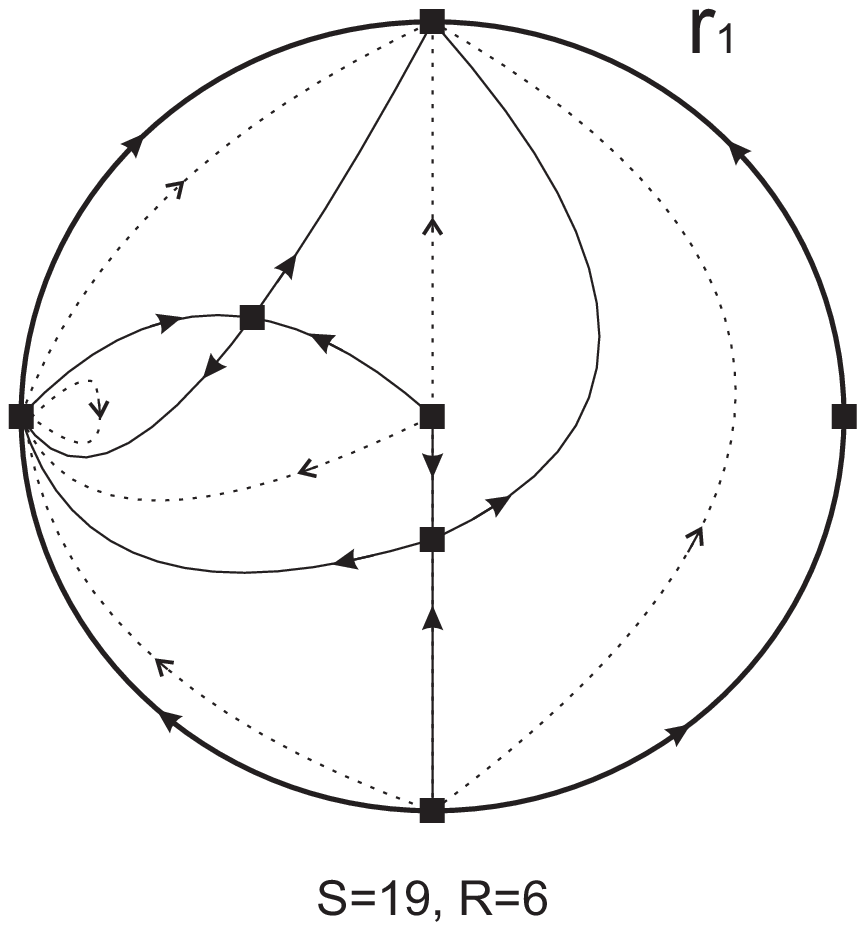,width=4.5cm,height=4.5cm} &
\epsfig{file=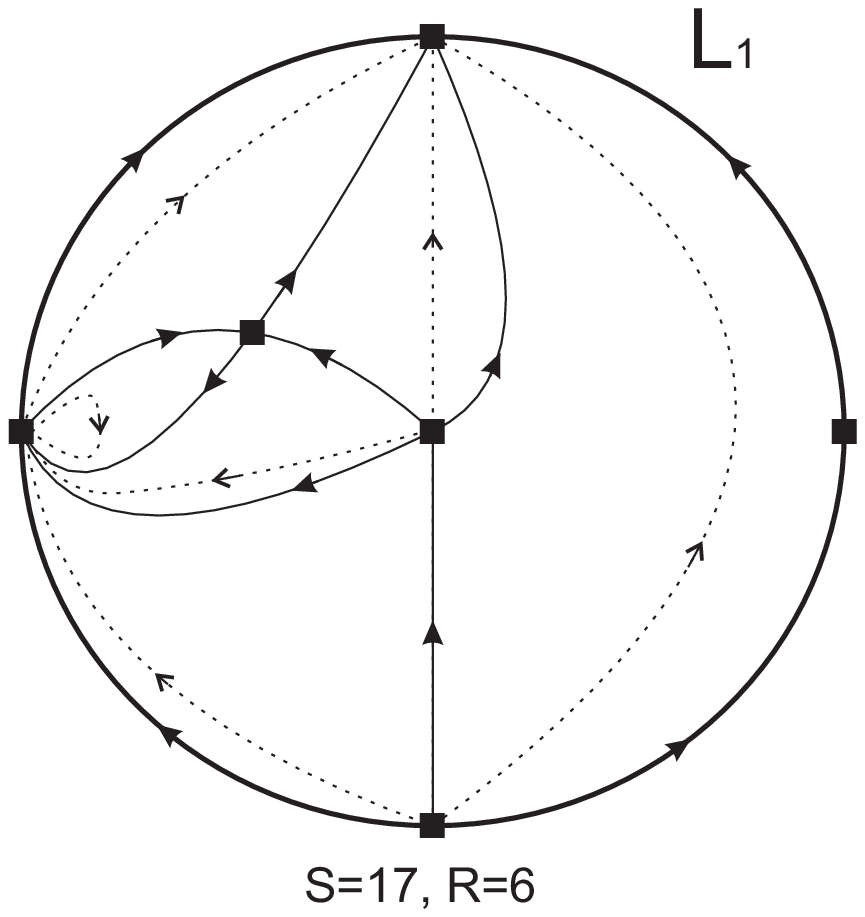,width=4.5cm,height=4.5cm} &
\epsfig{file=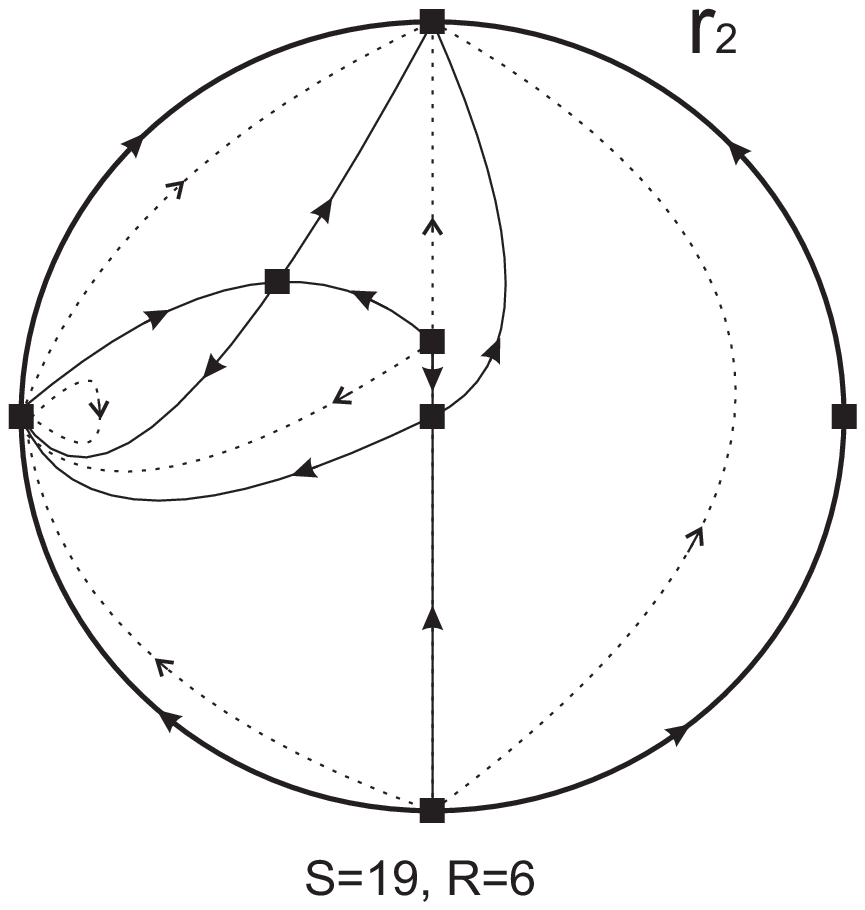,width=4.5cm,height=4.5cm} \\
	
\epsfig{file=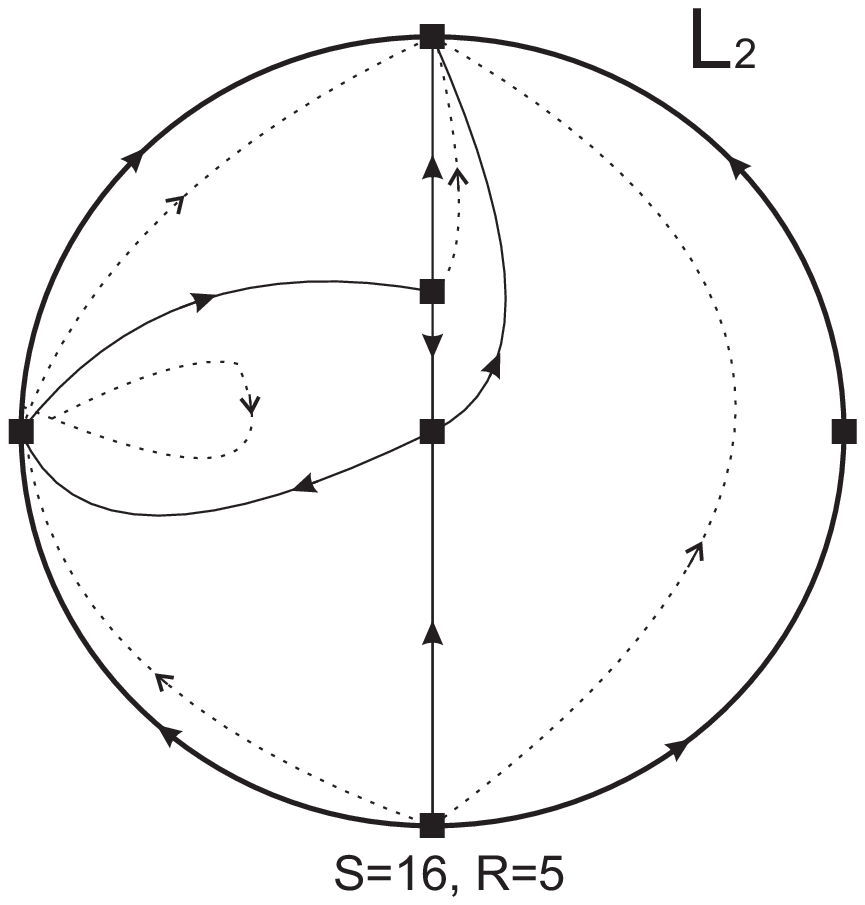,width=4.5cm,height=4.5cm} &
\epsfig{file=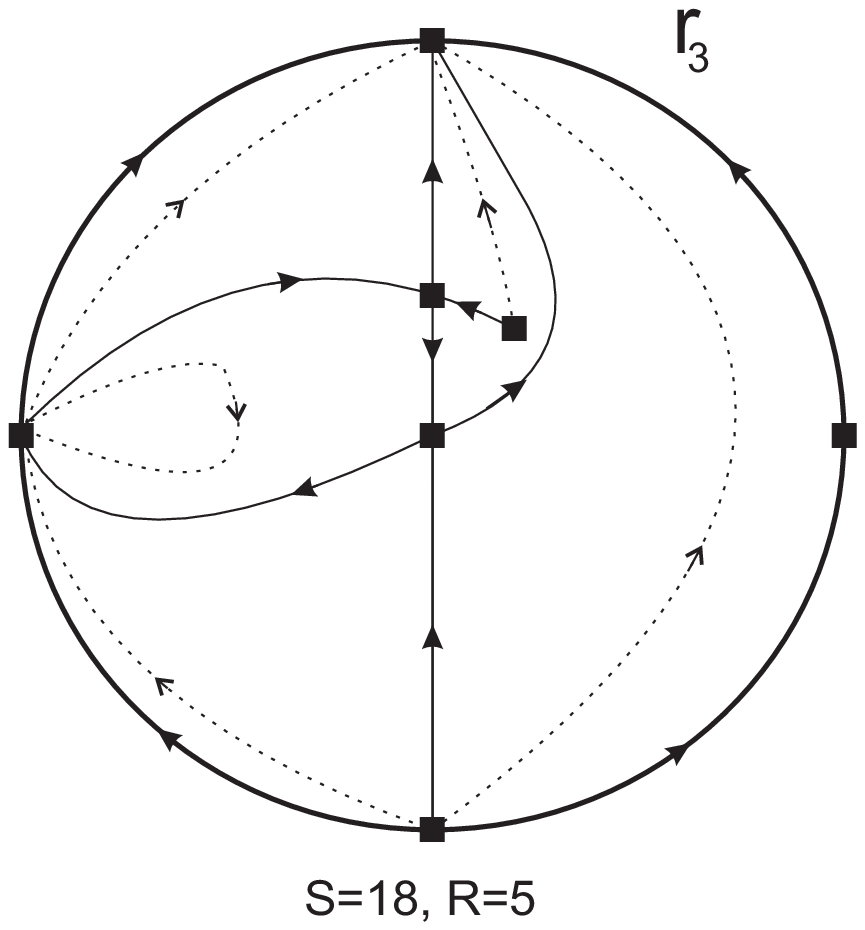,width=4.5cm,height=4.5cm}&
\epsfig{file=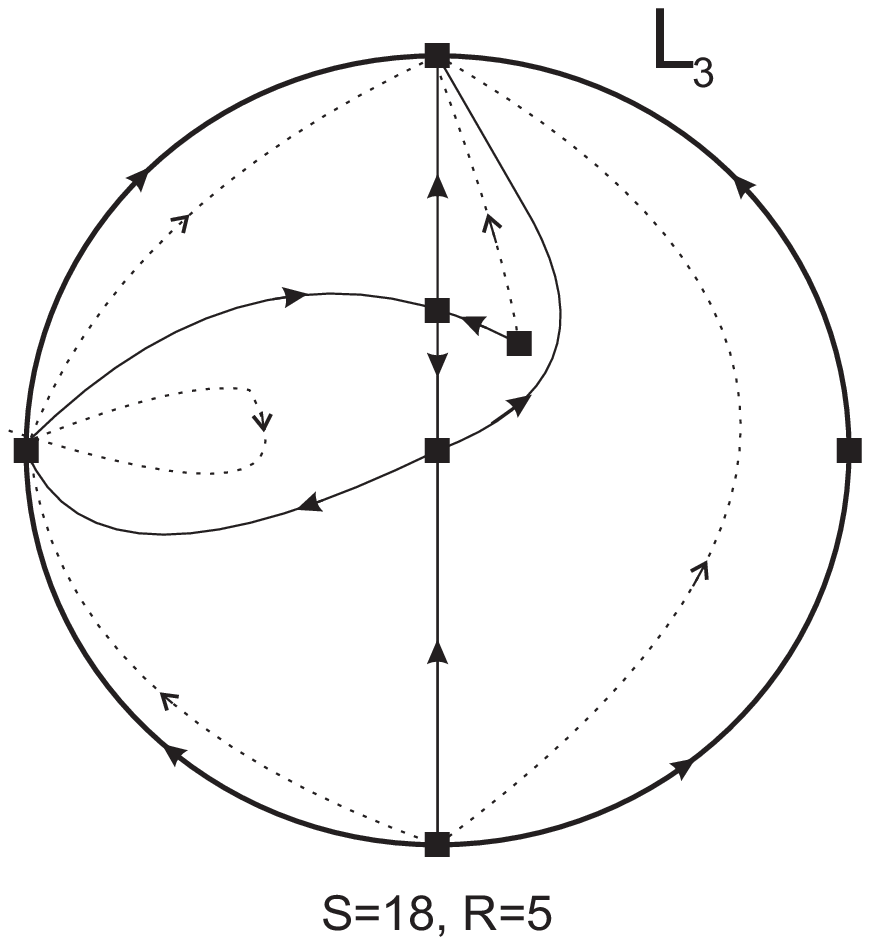,width=4.5cm,height=4.5cm}\\
	
\epsfig{file=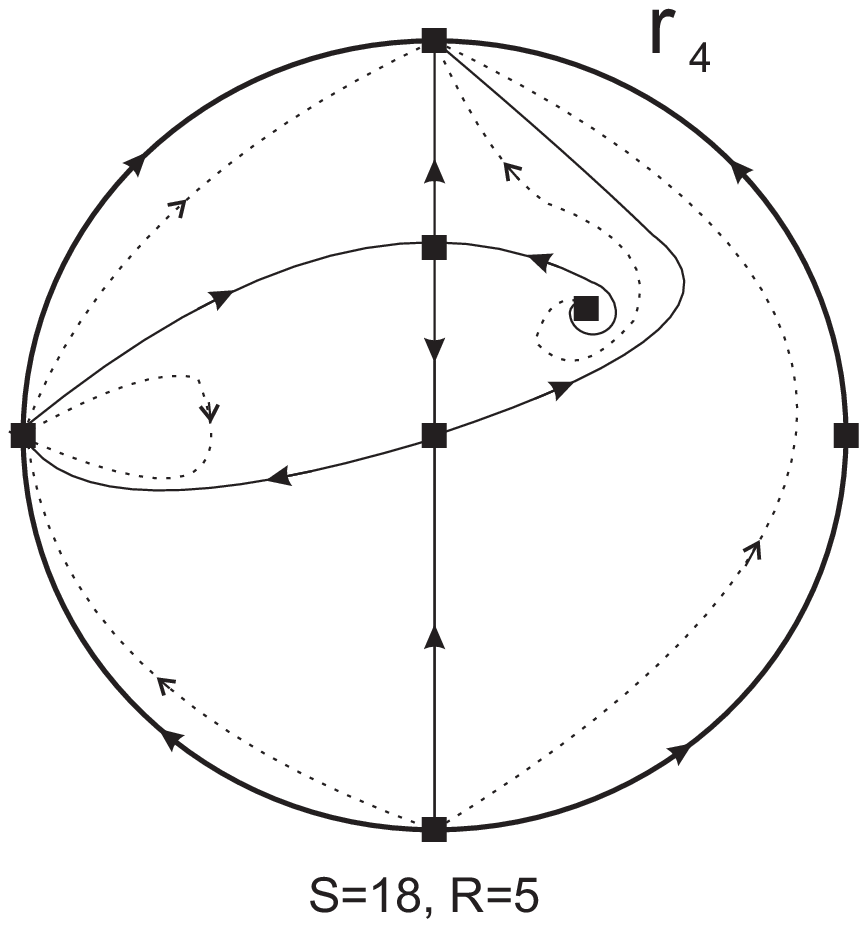,width=4.5cm,height=4.5cm} &
\epsfig{file=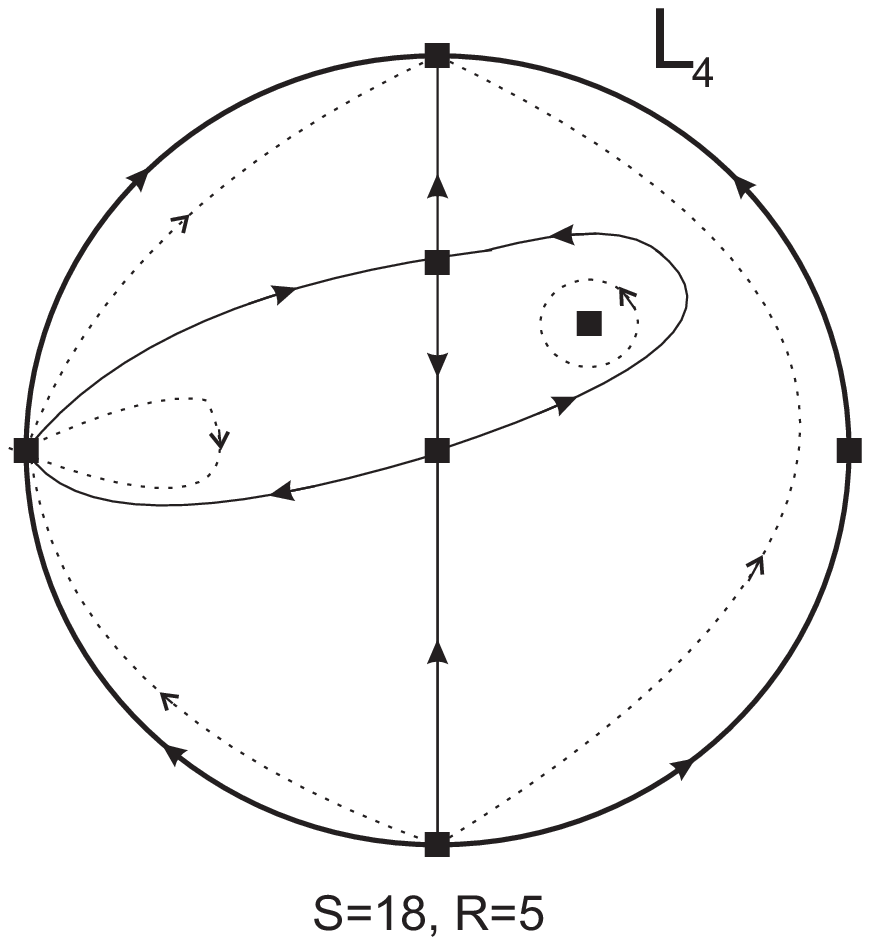,width=4.5cm,height=4.5cm}&
\epsfig{file=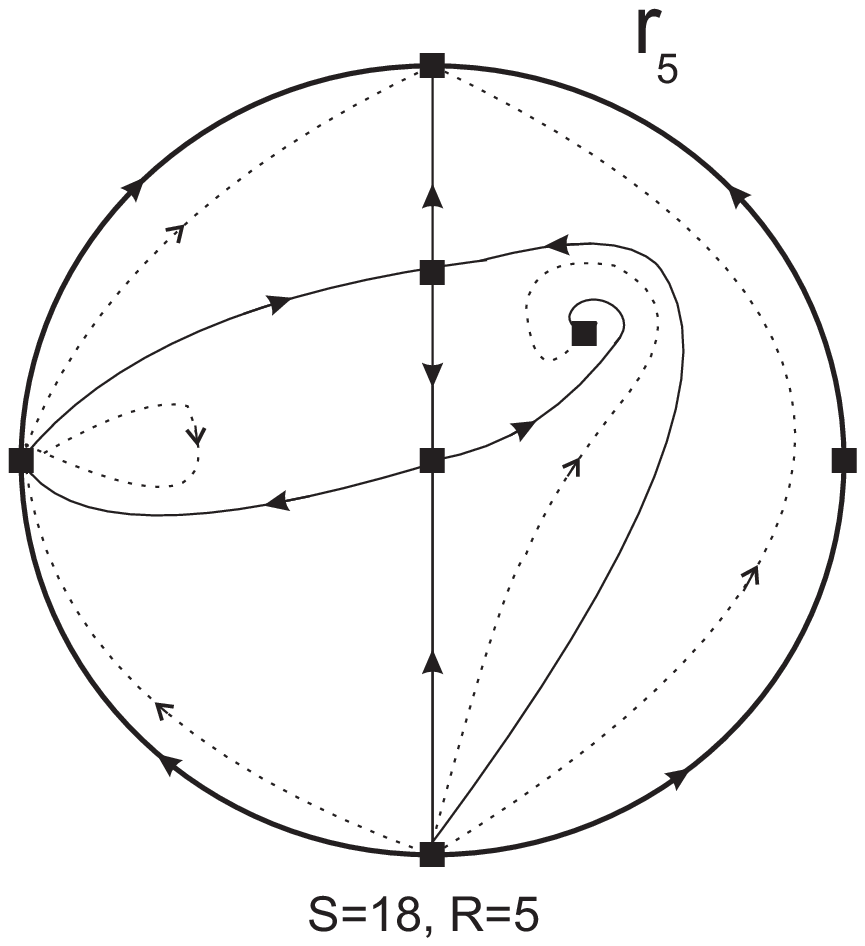,width=4.5cm,height=4.5cm}
\end{array}
$$
\caption{Global phase portraits of system \eqref{e2}.}
\label{f2}
\end{figure}

\begin{figure}[ht]
$$
\begin{array}{ccc}
\epsfig{file=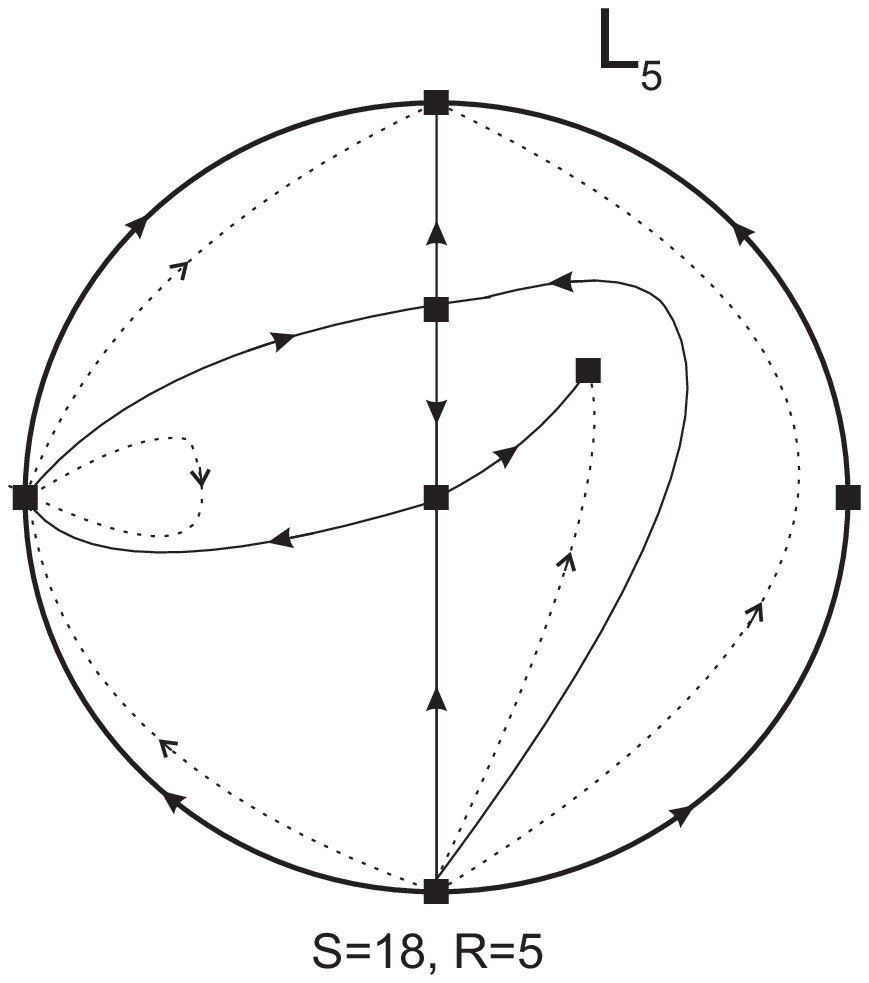,width=4.5cm,height=4.5cm} &
\epsfig{file=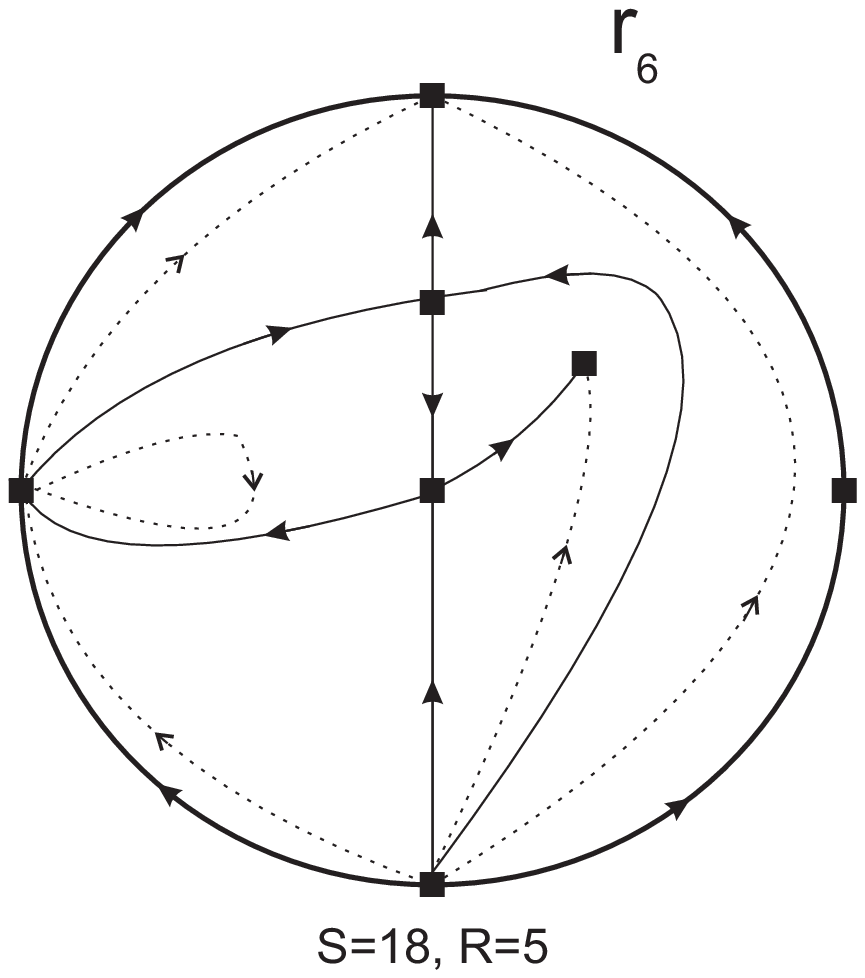,width=4.5cm,height=4.5cm }&
\epsfig{file=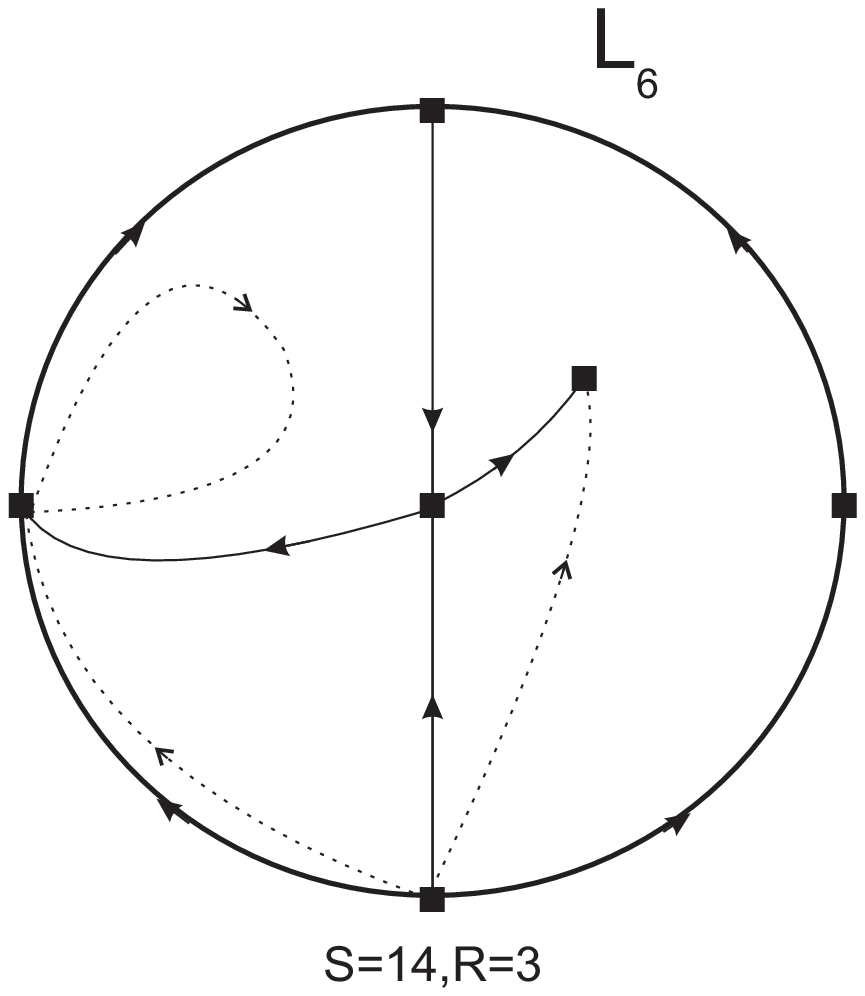,width=4.5cm,height=4.5cm}\\
	
\epsfig{file=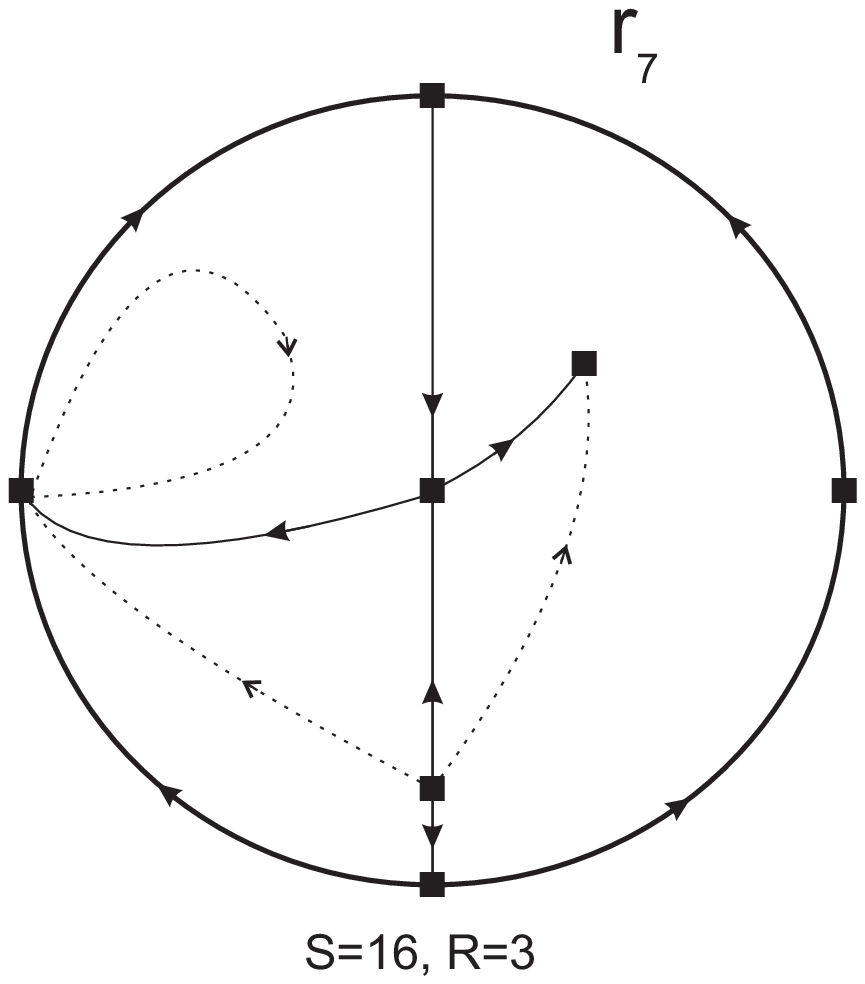,width=4.5cm,height=4.5cm} &
\epsfig{file=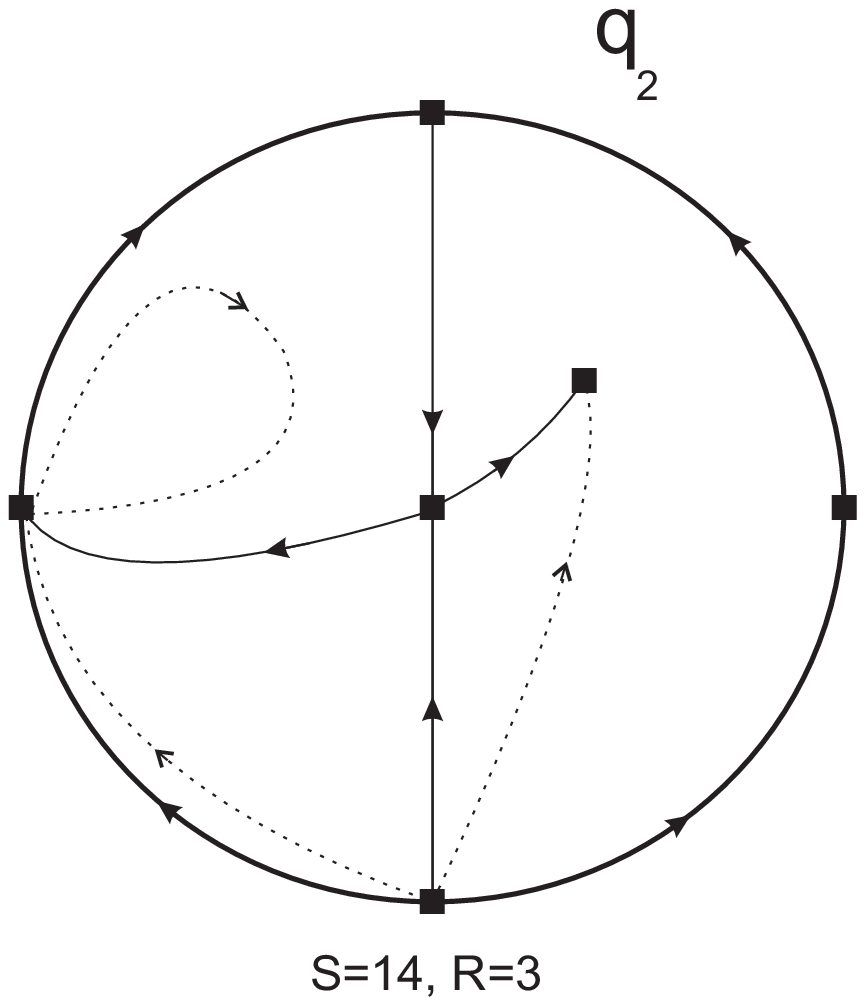,width=4.5cm,height=4.5cm }&
\epsfig{file=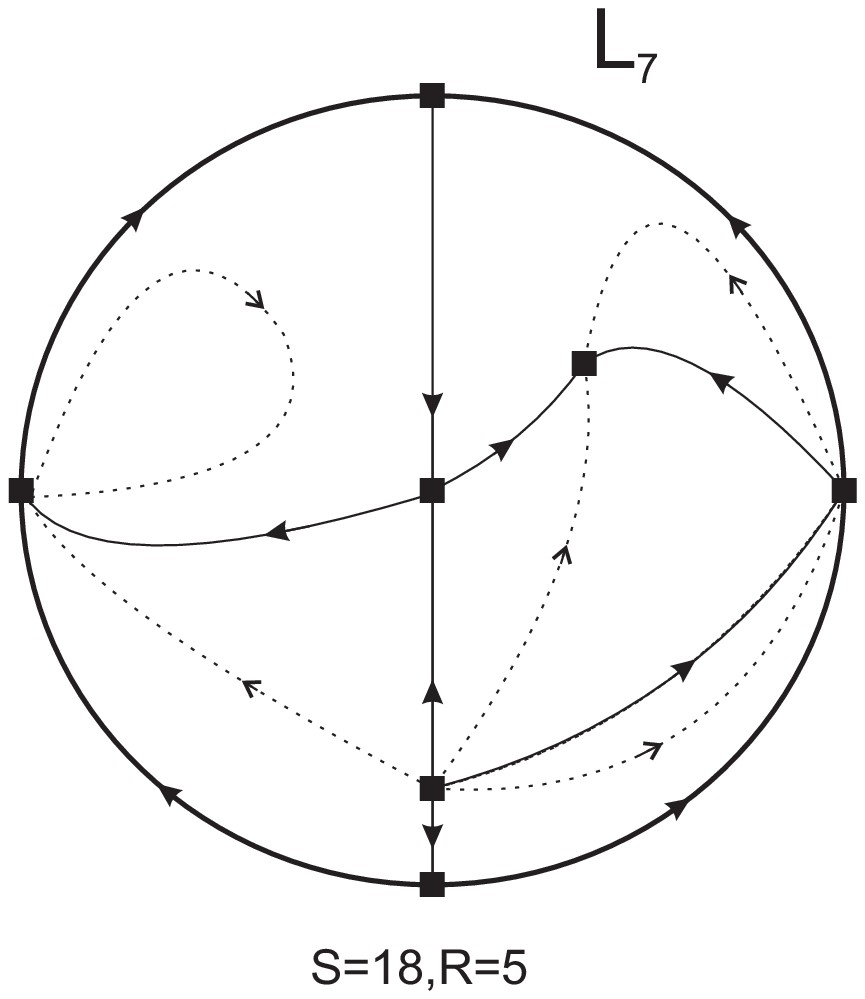,width=4.5cm,height=4.5cm}\\
	
\epsfig{file=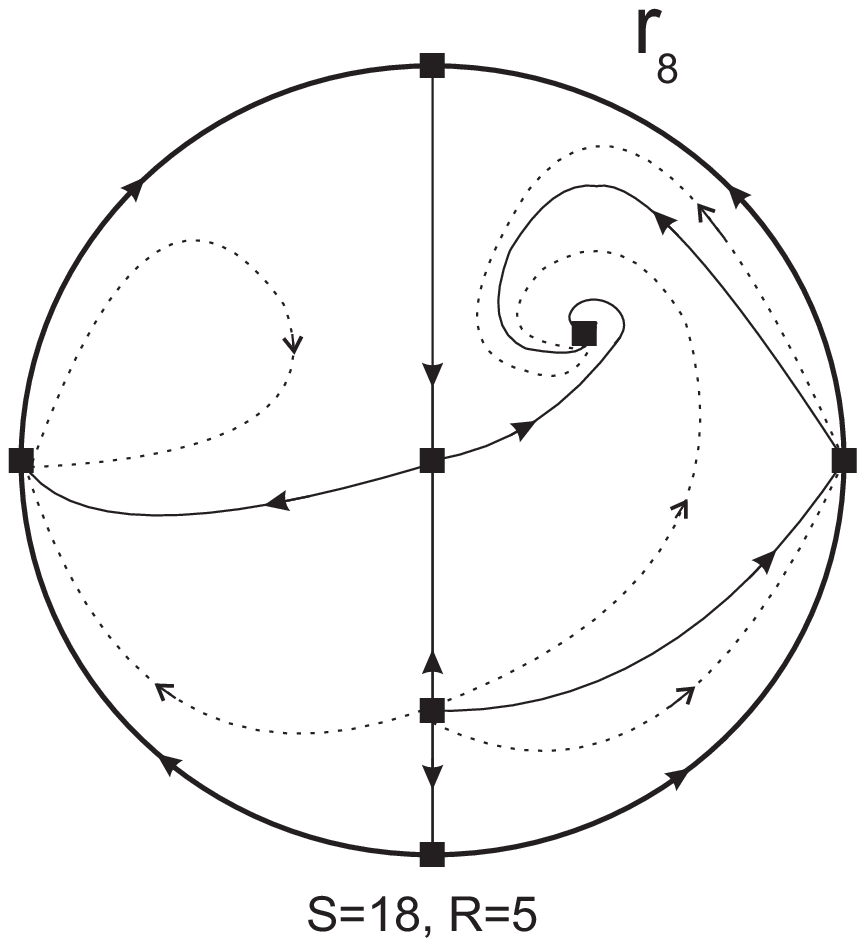,width=4.5cm,height=4.5cm} &
\epsfig{file=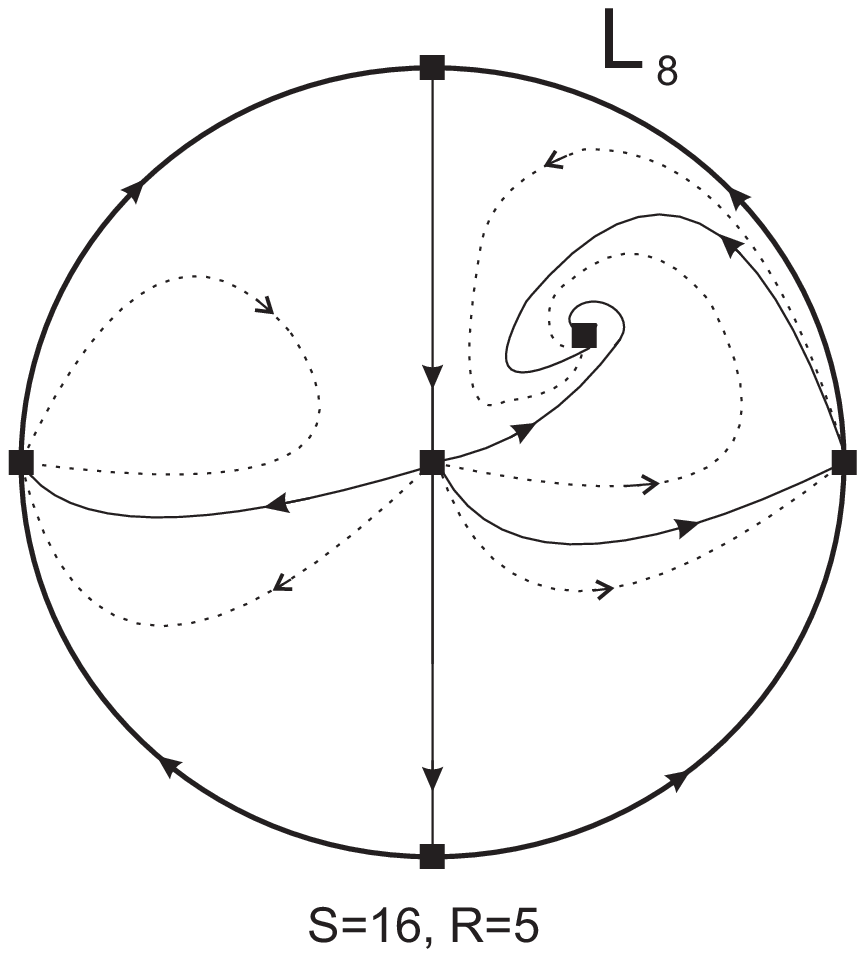,width=4.5cm,height=4.5cm} &
\epsfig{file=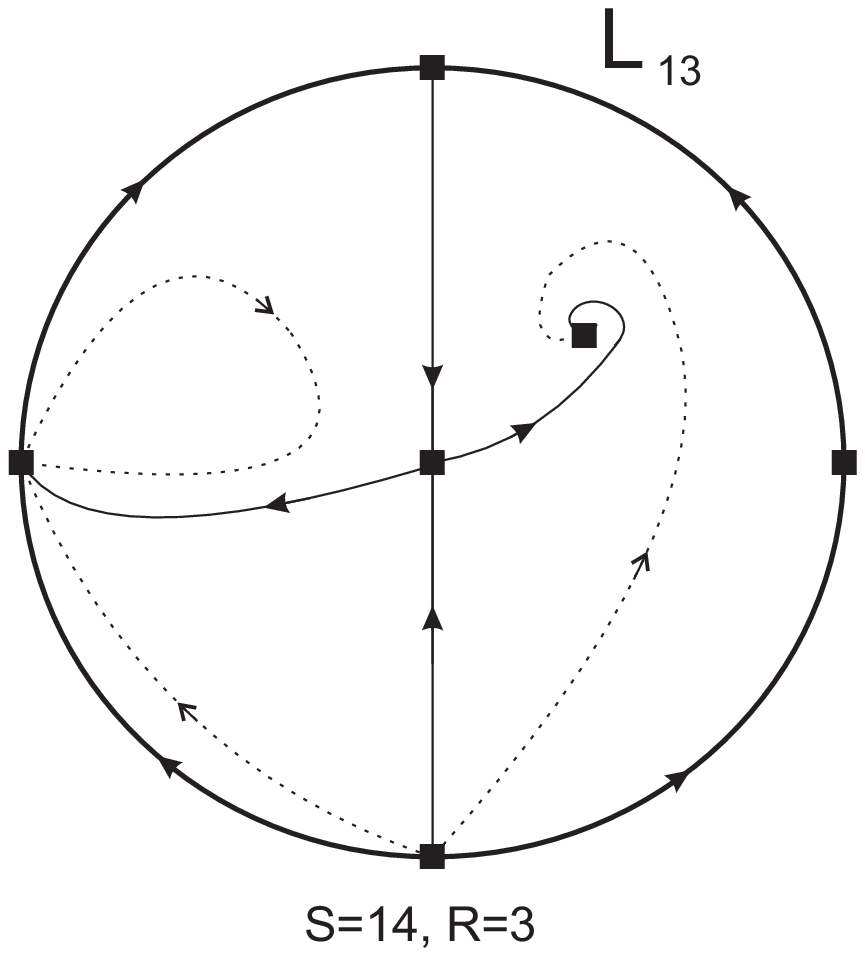,width=4.5cm,height=4.5cm}
\end{array}
$$
\caption{Global phase portraits of systems \eqref{e2}.}
\label{f3}
\end{figure}

\begin{figure}[ht]
$$
\begin{array}{ccc}
	
\epsfig{file=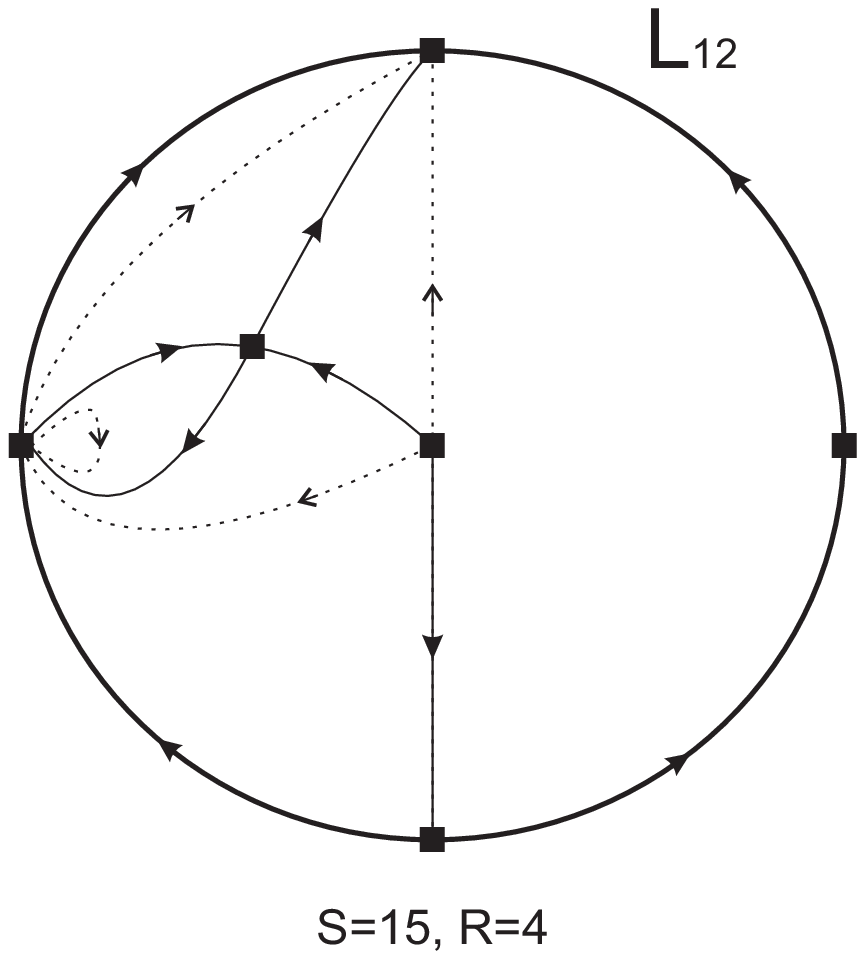,width=4.5cm,height=4.5cm} &
\epsfig{file=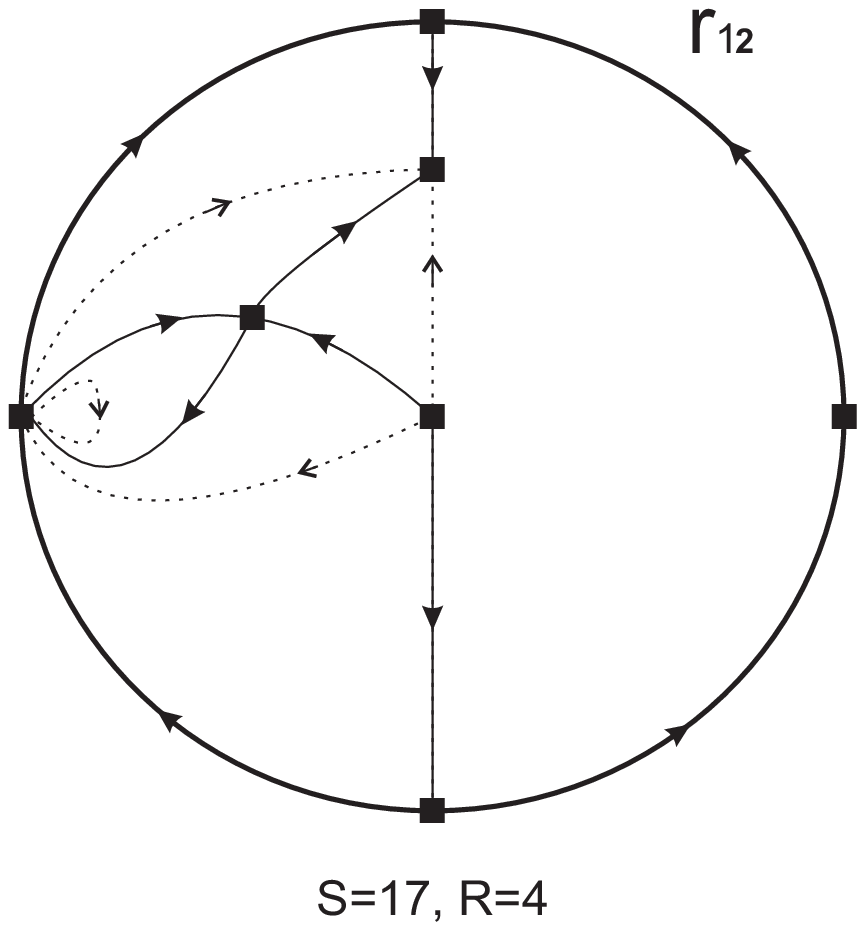,width=4.5cm,height=4.5cm} &
\epsfig{file=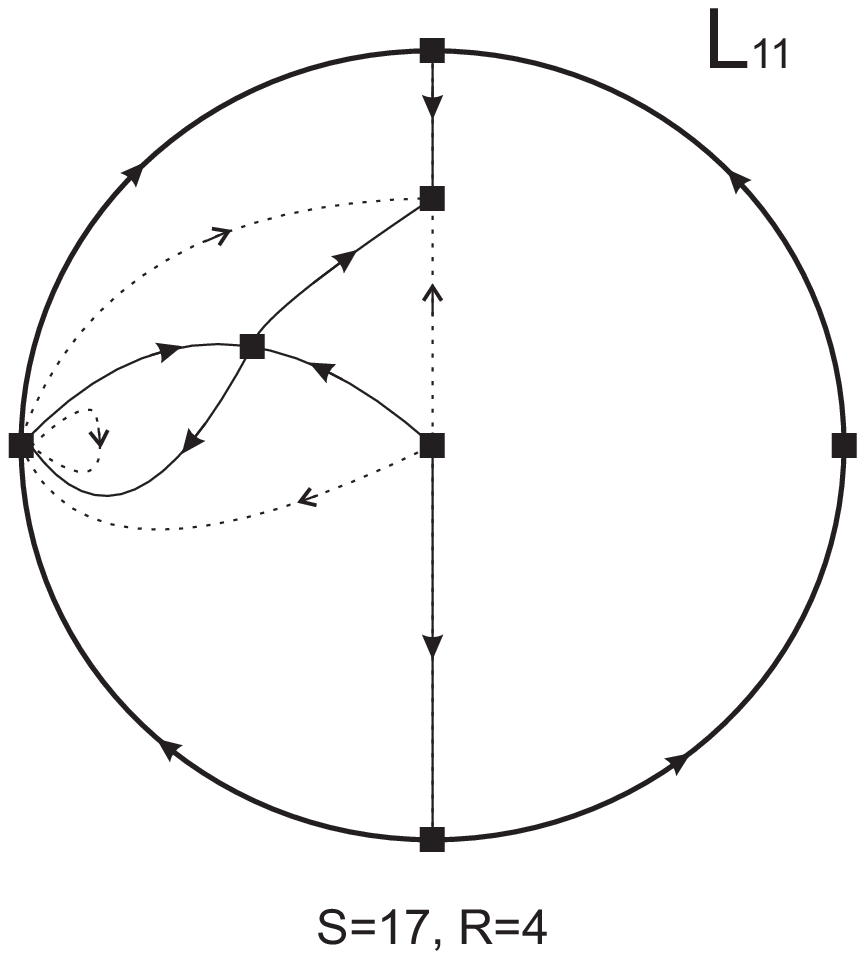,width=4.5cm,height=4.5cm} \\
	
\epsfig{file=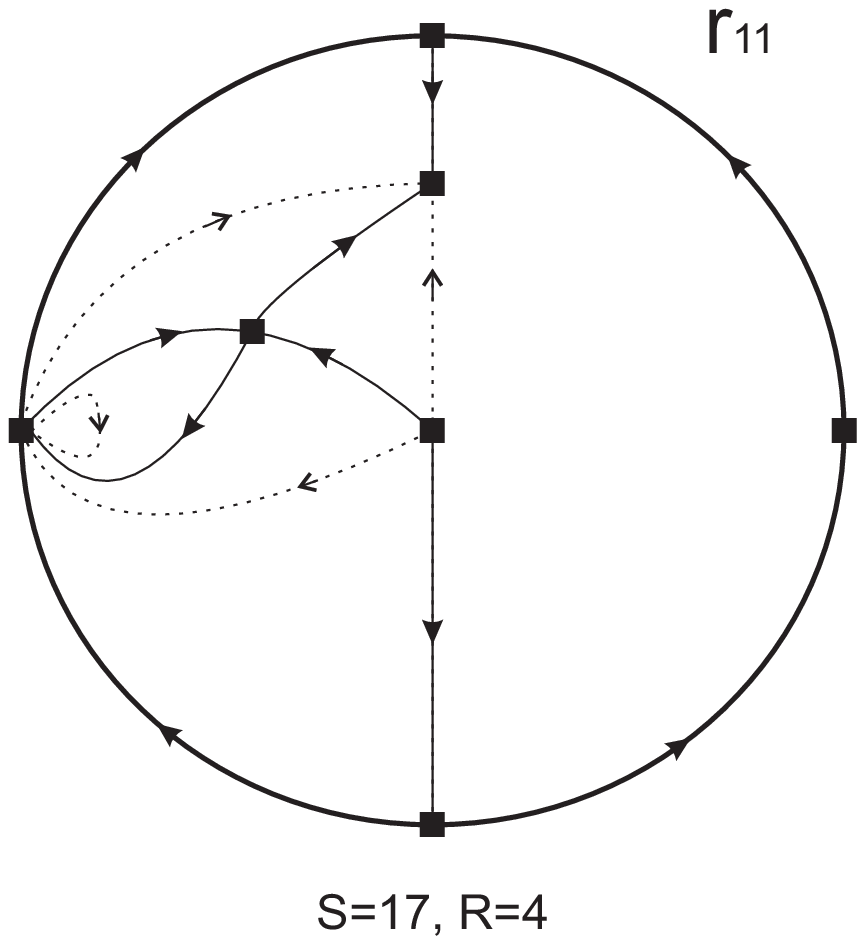,width=4.5cm,height=4.5cm} &
\epsfig{file=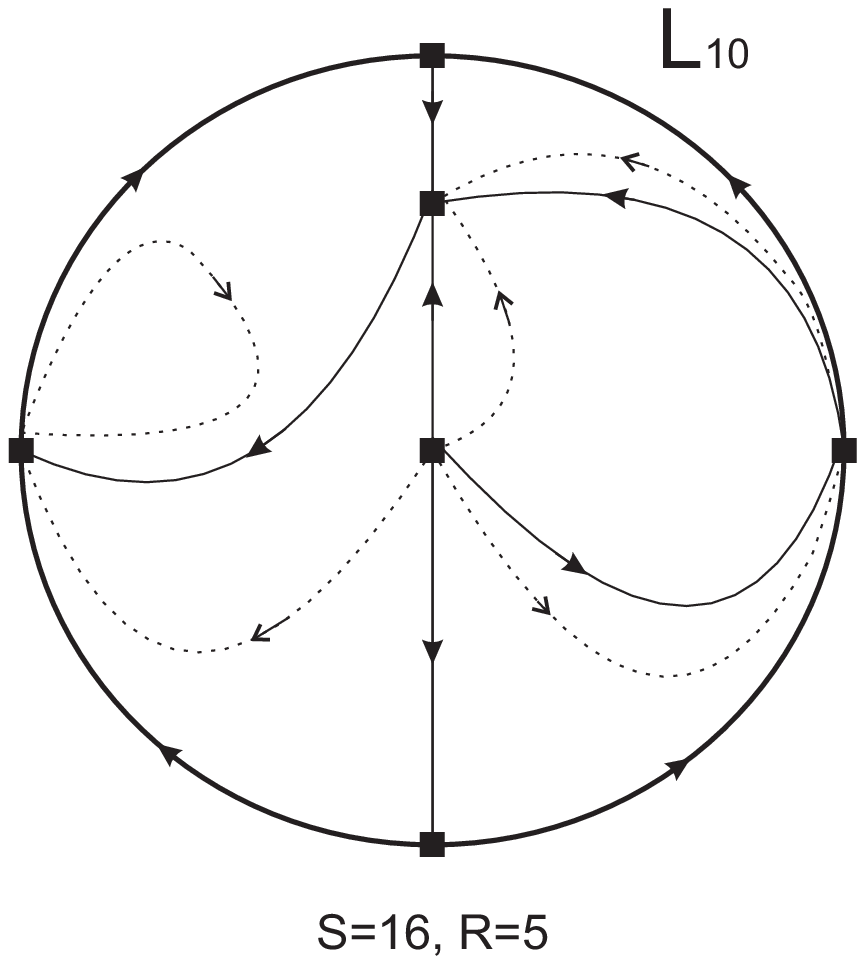,width=4.5cm,height=4.5cm} &
\epsfig{file=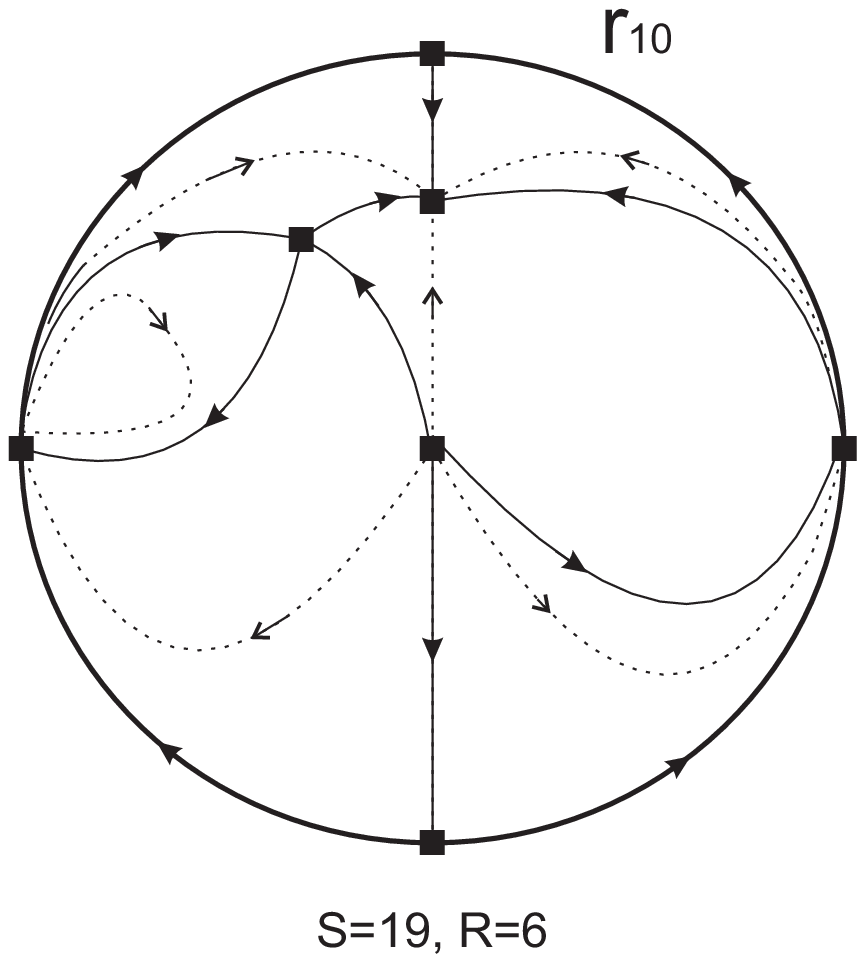,width=4.5cm,height=4.5cm} \\
	
\epsfig{file=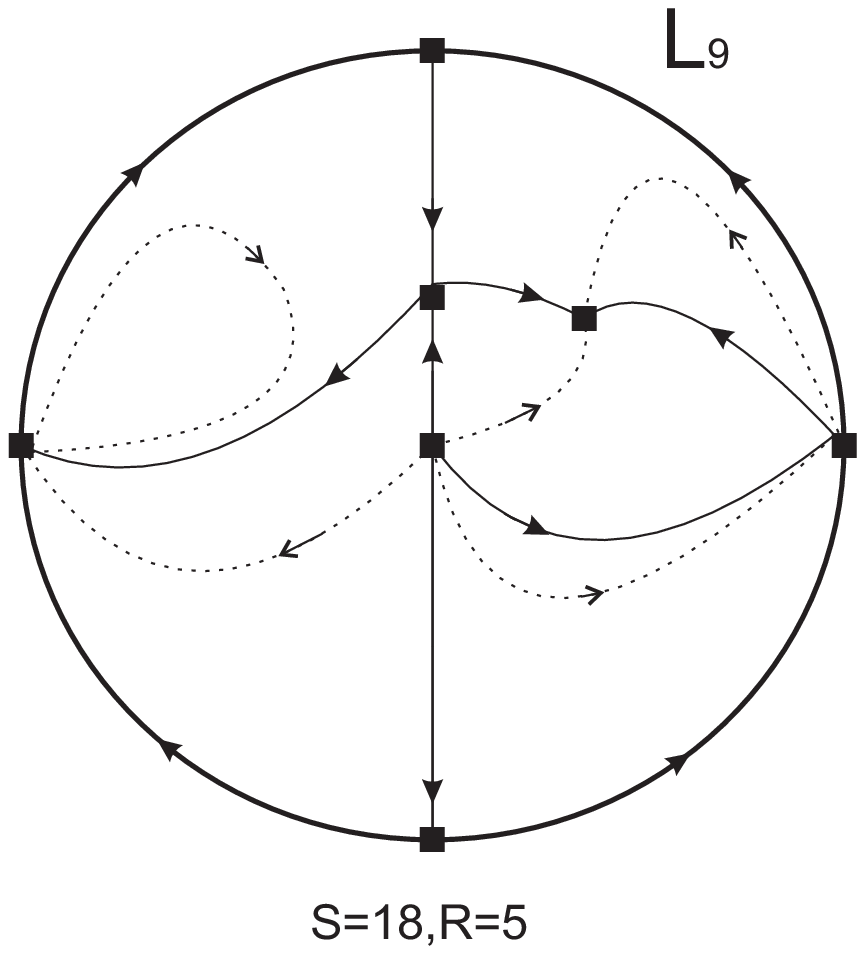,width=4.5cm,height=4.5cm} &
\epsfig{file=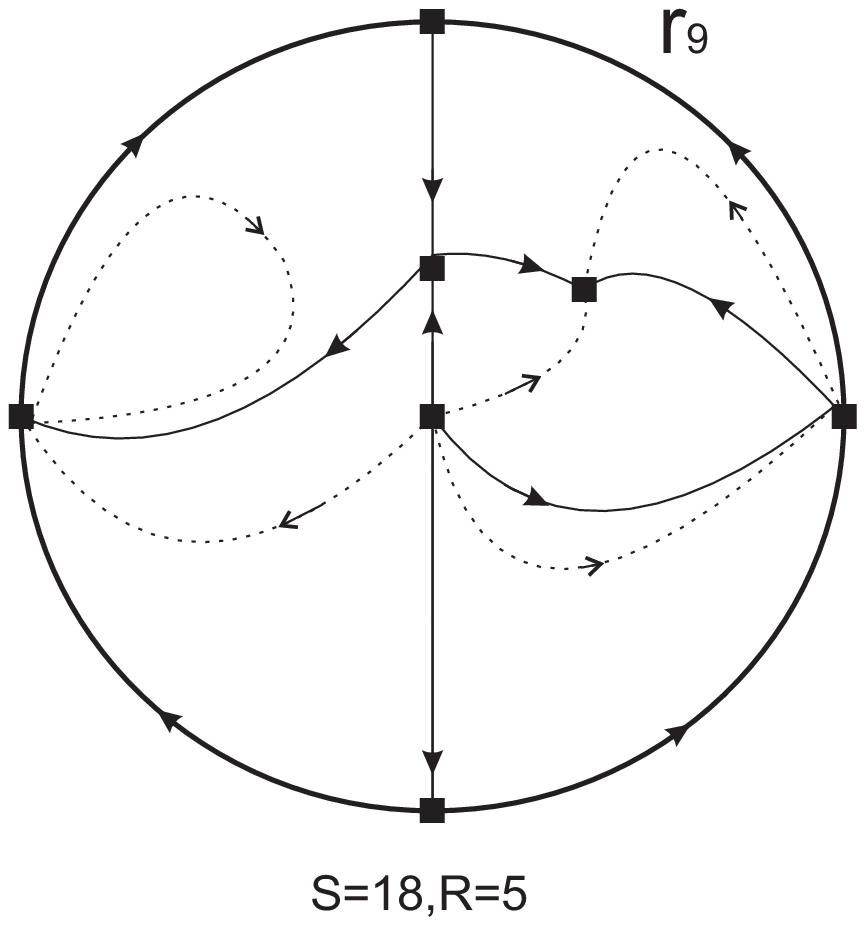,width=4.5cm,height=4.5cm} &
\epsfig{file=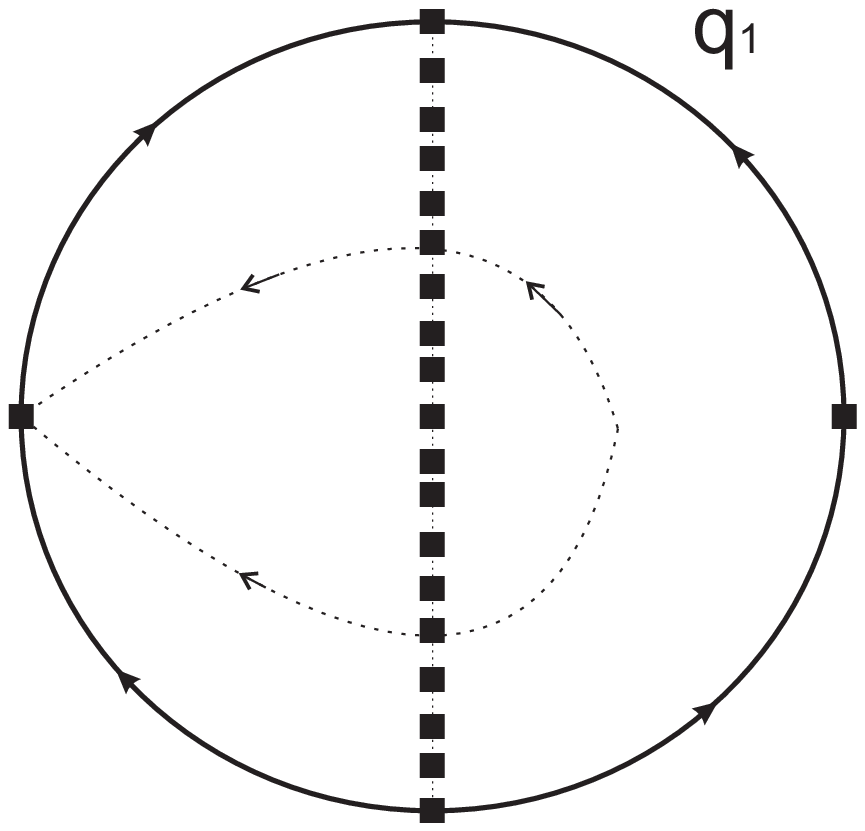,width=4.5cm,height=4.5cm}
\end{array}
$$
\caption{Global phase portraits of system \eqref{e2}.}
\label{f4}
\end{figure}

\begin{theorem}\label{t1}
The following statements hold.
\begin{itemize}
\item[(a)] For $c=2b+1$ system \eqref{e2} has the Darboux first integral
\begin{equation}\label{Darboux-first}
H(x,y)=x^{2b}\left(\left( 2\,b+1 \right) {y}^{2}-2\left( 2\,b+1 \right) y+2\,x\right).
\end{equation}
		
\item[(b)] For $c=1/2$ and  $b=-1/4$ (note that $c=2b+1$) system \eqref{e2} has the previous Darboux first integral and the Darboux invariants
\begin{equation}\label{I12}
I_1=x^{-1/2}(y^2+4x)\exp({-t/2}), \qquad I_2=x^{-1/2}[(y-2)^2+4x]\exp({t/2}).
\end{equation}
		
\item[(c)] For $b=(1-c)/(2c-3)$ system \eqref{e2} has the Darboux invariant
\begin{equation}\label{Dinv2}
I_3=x^{2(1-c)/(2c-3)}\left[\left(y-(3-2c)\right)^2+2(3-2c)x\right]\exp\left(\frac{2(1-c)}{3-2c}t\right).
\end{equation}
\end{itemize}
\end{theorem}

Theorem \ref{t1} is proved in Section \ref{s2}.

In this paper all the phase portraits are drawn in the so called Poincar\'e disc, which roughly speaking, is the closed disc centered at the origin of coordinates and of radius one. The interior of this disc is identified with $\R^2$. Its boundary, the circle $\mathbb S^1$, is identified with the infinity of $\R^2$. In the plane $\R^2$ we can go to infinity in as many directions as points in the circle $\mathbb S^1$. For more details about the Poincar\'e disc and the coordinates for studying the polynomial differential system \eqref{e2} in it, see Appendix 1.

\begin{theorem}\label{t2}
The Basener--Ross differential system \eqref{e2} has $15$ non-topological equivalent phase portraits in the Poincar\'e disc $($see Figures $\ref{f2}$, $\ref{f3}$ and $\ref{f4}$ and the proof of this theorem$)$.
\end{theorem}

Theorem \ref{t2} is proved in Section \ref{s3}.

\section{Darboux integrability}\label{s2}

In this section we prove Theorem \ref{t1}.

We denote by
\begin{equation}\label{vf}
\X=x(1-y){\frac{\partial }{\partial x}}+(by^2+(1-c)y+x){\frac{\partial }{\partial y}}
\end{equation}
the vector field defined by system \eqref{e2} .

The straight line $f_1=x=0$ is an invariant algebraic curve of the vector field \eqref{vf} with cofactor $K_1=1-y$.

First consider $c=2b+1\neq 0$. Then system \eqref{e2}   admits  the invariant algebraic curve
$$
f_2= \frac{2b+1}{2} y^2- (2b+1)y+x=0,
$$
with cofactor $K_2=2b(y-1)$. Note that $2bK_1+K_2=0$, so from statement (a) of Theorem \ref{t0} we have that for $c=2b+1$ system \eqref{e2} admits the Darboux first integral $H=f_1^{2b} f_2$ given in \eqref{Darboux-first}. Therefore statement (a) of Theorem \ref{t1} is proved.

Now we consider the special choice of the parameters $c=1/2$ and $b=-1/4$. Note that still relation $c=2b+1$ holds, and consequently system \eqref{e2} has the Darboux first integral \eqref{Darboux-first} for these values of the parameters. Furthermore, $f_3=y^2+4x=0$ and $f_4=(y-2)^2+4x=0$ are invariant algebraic curves of system \eqref{e2} with cofactors $K_3=(2-y)/2$ and $K_4=-y/2$, respectively. Since $$
-\frac12 K_1+K_3= \frac12 \quad  \mbox{and} \quad -\frac12 K_1+K_4= -\frac12,
$$
from statement (b) of Theorem \ref{t0} it follows that $I_1$ and $I_2$ are Darboux invariants of system \eqref{e2}. This completes the proof of statement (b) of Theorem \ref{t1}.

Finally we consider system \eqref{e2} with $b=(1-c)/(2c-3)$. Then $f_5= 2 (3 - 2 c) x + (-3 + 2 c + y)^2=0$ is an invariant algebraic curve of this system with cofactor $K_5= 2(c-1)y/(3-2 c)$. Since
$$
\frac{2(1-c)}{2c-3} K_1 + K_5= \frac{2(c-1)}{3-2 c}.
$$
Again from statement (b) of Theorem \ref{t0} we get that $I_3$ is a Darboux invariant of system \eqref{e2}. So statement (c) of Theorem \ref{t1} is proved.

\section{The global phase portraits in the Poincar\'e disc of system  \eqref{e2}}\label{s3}

In order to present the global phase portraits of the two-parametric family  \eqref{e2} we first study its finite singular points, in Subsection \ref{sfinite} and after the infinite singular points in Subsection \ref{sinfinite}.

\subsection{The Finite singular points of system \eqref{e2}}\label{sfinite}

System \eqref{e2} has the following finite singular points whenever they are defined (i.e. if $b\ne 0$)
\begin{equation}\label{sing-points}
P_0=(0,0), \quad P_1=(0,(c-1)/b), \quad P_2=(c-b-1,1).
\end{equation}

The origin  $P_0$ has eigenvalues $1$ and $1-c$. So $P_0$ is a hyperbolic unstable node for $c<1$ and a hyperbolic saddle for $c>1$. For $c=1$ the origin $P_0$ collapses to the point $P_1$ and becomes a semi--hyperbolic saddle node, see Theorem 2.19 of \cite{DLA}. Hence we consider the bifurcation curve $g_1=c-1=0.$

The point $P_1$ is defined for $b\neq 0$ and has eigenvalues $c-1$ and $(b-c+1)/b$. We consider the bifurcation curves $g_0=b=0,$ and $g_2=b-c+1=0.$ Note that for our study we always have $g_0>-1$. Then for $b>0$, $g_1>0$ and $g_2>0$ the point $P_1$ is a hyperbolic unstable node, and for $b>0$, $g_1<0$ and $g_2>0$ it is a hyperbolic saddle. Then for $b<0$, $g_1>0$ and $g_2>0$ it is a hyperbolic saddle, and for $b<0$, $g_1<0$ and $g_2>0$ it is a hyperbolic stable node. Then for $b>0$, $g_1>0$ and $g_2<0$ it is a hyperbolic saddle, and for $b>0$, $g_1<0$ and $g_2<0$ it is a hyperbolic stable node. Then for $b<0$, $g_1>0$ and $g_2<0$ it is a hyperbolic unstable node, and for $b<0$, $g_1<0$ and $g_2<0$ it is a hyperbolic saddle. For $c=b+1$ we have that $g_2=0$ and the point $P_1$ collapses to the point $P_2$ and it becomes a semi--hyperbolic saddle--node, see again Theorem 2.19 of \cite{DLA}.

The point $P_2$ has the eigenvalues
$$
\lambda_{\pm}=\dfrac{2b-c+1}{2}\pm\frac{\sqrt{D_1}}{2},
$$
with $D_1={c}^{2}-4\,cb+4\,{b}^{2}-6\,c+8\,b+5=0$ a parabola in the $(b,c)$-plane and it is a bifurcation curve. Note that $\lambda_- \lambda_+= c-b-1=-g_2.$ First consider $D_1\geq 0$.   Then for $g_2<0$ the point $P_2$ is a hyperbolic node, whereas for $g_2>0$ is a hyperbolic saddle. For $g_2=0$ we recall that the point $P_2$ collapses to the point $P_1$ and it is a semi--hyperbolic saddle node.

Now we consider $D_1<0.$ Additionally we must consider the curve $g_3=2b-c+1=0$. For $g_3>0$ the point $P_2$ is a hyperbolic unstable focus, whereas for $g_3<0$ we have that $P_2$ is a hyperbolic stable focus. For $g_3=0$  the point $P_2$ could be a focus or a  center. Since for $c=2b+1\neq 0$ there is the first integral \eqref{Darboux-first} defined at $P_2$, it follows that $P_2$ is a center.

\subsection{The infinite singular points of system \eqref{e2}}\label{sinfinite}

In order to describe the global phase portraits of system \eqref{e2} in the Poincar\'e disc we must study the infinite singular points, see Appendix 1.

System \eqref{e2} in the chart $(U_1,F_1)$ is written
\begin{equation}\label{uf1}
\begin{array}{ccl}
\dot{z_1}&=&{z_2}+(b+1)z_1^{2}-c{z_1}\,{z_2},\\
\dot{z_2}&=&{z_2}\, \left(z_1 -{z_2}\right),
\end{array}
\end{equation}
and the origin $(0,0)$ of $(U_1, F_1)$ is a singular point. Note that the linear part of system \eqref{uf1} at the origin is not identically zero but its two eigenvalues are equal to zero.  So the origin is a nilpotent singular point. Applying Theorem 3.5 of \cite{DLA} we have that for $b>-1$ the origin $(0,0)$ of the chart $(U_1,F_1)$ is the union of one hyperbolic and one elliptic sector.

System \eqref{e2} in the chart $(U_2,F_2)$ is
\begin{equation}\label{uf2}
\begin{array}{ccl}
\dot{z_1}&=&{z_1}\, \left( c{z_2}-{z_1}\,{z_2}-b-1 \right), \\
\dot{z_2}&=&{z_2}\, \left( c{z_2}-{z_1}\,{z_2}-{z_2}-b \right),
\end{array}
\end{equation}
and the origin is the only singular point with eigenvalues $-b-1$, $-b.$  We consider the bifurcation curve $g_4=b+1=0.$  Since $b>-1$ we have that the origin of the chart $(U_2,F_2)$ is a hyperbolic saddle for $-1<b<0,$ and  for $b>0$ is a hyperbolic stable node. For $b=0$ it is a semi--hyperbolic saddle--node.

\begin{center}
\begin{figure}
\epsfig{file=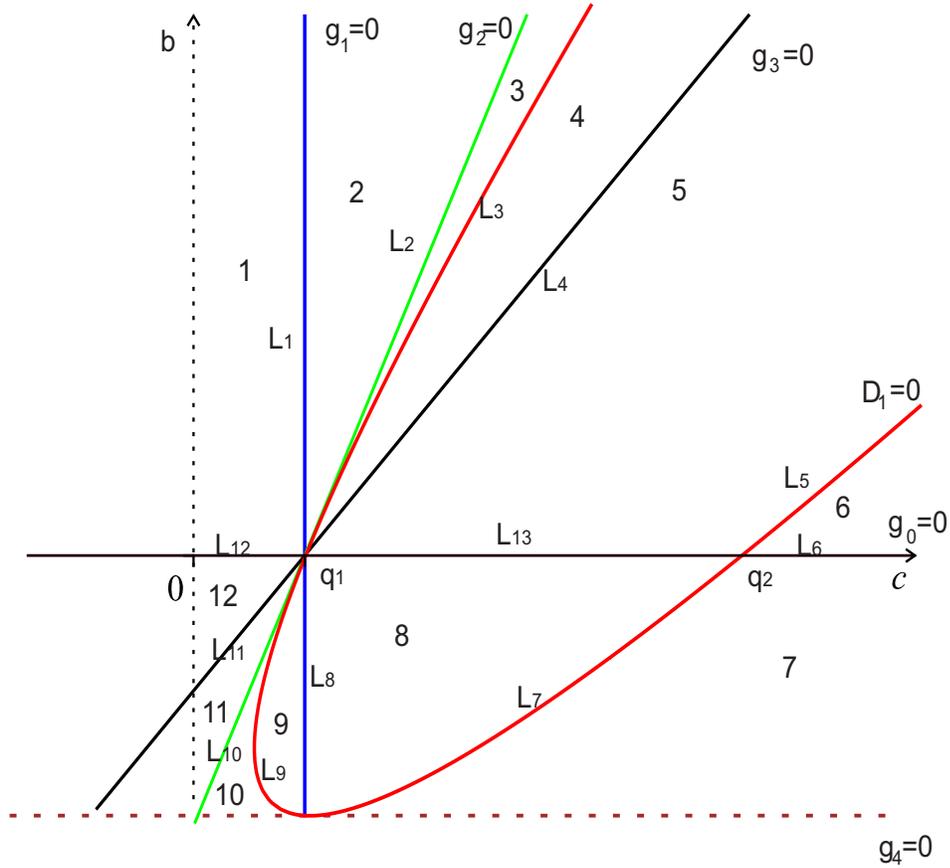,width=12.5cm,height=11.5cm}
\caption{The bifurcation diagram of system \eqref{e2} with  $c>0$ and $b>-1$.}
\label{figurediagrama}
\end{figure}
\end{center}

\subsection{The topological classification of the global phase portraits of system \eqref{e2}}

From Subsections \ref{sfinite} and \ref{sinfinite} we obtain the bifurcation diagram of system \eqref{e2} given in Figure \ref{figurediagrama}.

In the bifurcation diagram of Figure \ref{figurediagrama} we have twelve regions $r_1,\ldots,r_{12},$ the thirteen lines $L_1,\ldots,L_{13}$, and two points $q_1=(1,0)$ and $q_2=(5,0)$.

Using that the straight line $x=0$ of system \eqref{e2} is invariant, how is the flow of system \eqref{e2} on $y=0$, and the local phase portraits of the finite and infinite singular points of system \eqref{e2} we obtain the phase portraits of system \eqref{e2} described in the Figures \ref{f2}, \ref{f3} and \ref{f4} according with the twelve regions, the thirteen lines and the two points. In these figures we denote by $S$ the number of separatrices and by $R$ the number of the canonical regions, see for more details Appendix 2.

In Figures \ref{f2}, \ref{f3} and \ref{f4} are claimed that the phase portraits of system \eqref{e2} has no limit cycles. Now we prove the claim. First it is known that if a quadratic polynomial differential system has a limit cycle this limit cycle must surround a focus, see Theorem 6 of \cite{coppel}. Also it is known that if a quadratic polynomial differential system has an invariant straight line it has at most one limit cycle, which must be stable or unstable. The first proof of this result is due to Ryckov \cite{Ry}, a more clear proofs appear later on in \cite{CL, C2}. From the phase portraits of Figures \ref{f2}, \ref{f3} and \ref{f4} it follows that if there is a focus surrounded by a limit cycle, this limit cycle must be semistable, but since system \eqref{e2} has the invariant straight line $x=0$ if it has a limit cycle this must be unique and either stable or unstable, hence system \eqref{e2} cannot have limit cycles. The claim is proved.

\begin{proof}[Proof of Theorem $\ref{t2}$]
We recall that two global phase portraits are topological equivalent if and only if does exist a homeomorphism to bring the separatrix configuration of one phase portrait into the separatrix configuration of the other, see Theorem \ref{mnp} in Appendix 2.

We distinguish the following $15$ topologically different phase portraits in the Poincar\'e disc using the mentioned Theorem \ref{mnp}.

Case $S=14, R=3$. We have that the phase portraits of the lines $L_6$ and $L_{13}$, and of the point $q_2$ are topological equivalents, and we simply write $L_6=L_{13}=q_2$.

Case $S=15, R=4$. We obtain the unique phase portrait $L_{12}$.

Case $S=16, R=3$. There is only the phase portrait $r_7$.

Case $S=16, R=5$. We have three different topological phase portraits $L_2$, $L_{8}$ and $L_{10}$.

Case $S=17, R=4$. We obtain a unique phase portrait $r_{12}=L_{11}=r_{11}$.

Case $S=17, R=6$. There is only the phase portrait $L_1$.

Case $S=18, R=5$. We have four different topological phase portraits $r_3=r_4=L_3$, $r_5=r_6=L_{5}$, $r_8=r_9=L_9=L_7$ and $L_4$.

Case $S=19, R=6$. We obtain two different topological phase portraits $r_1=r_2$ and $r_{10}.$

Case $q_1$.

This completes the proof of the theorem.
\end{proof}

\section{Appendix 1: Poincar\'e compactification}

We consider the quadratic polynomial differential system \eqref{e2} and its corresponding vector field $\X=(P,Q)$. We want to obtain the global phase portrait of system \eqref{e2}, and consequently we need to control the orbits that come from or escape to infinity. For this reason we use the so called {\it Poincar\'e compactification}, see Chapter 5 of \cite{DLA}.

Let $\R^2$ be the plane in $\R^3$ defined by $(y_1, y_2, y_3) = (x_1,x_2, 1)$. Consider the {\it Poincar\'e sphere} $\aS^2 = \{y=(y_1, y_2, y_3) \in \R^3 : y_1^2+ y_2^2 + y_3^2 = 1\}$  and we denote by $T_{(0,0,1)}\aS^2$ the tangent space to $\aS^2$ at the  point $(0,0,1)$ (see also \cite{point}). We consider the central projection $f: T_{(0,0,1)}:\R^2\rightarrow \aS^2$. Note that $f$ defines two copies of $\X$, one in the northern hemisphere $\{y\in \aS^2 : y_3 > 0\}$ and the other in the southern hemisphere. Now set $\hat{\X}=Df\circ \X$. We observe that $\hat{\X}$ is defined on $\aS^2$ except on its equator $\aS^1$. Hence, the points at infinity of $\R^2$ are in bijective correspondence with $\aS^1 = \{y \in \aS^2 : y_3 = 0\}$, (the equator of $\aS^2$). So according to this construction $\aS^1$ is identified to be the infinity of $\R^2$. The resulting {\it Poincar\'e compactified vector field} $p(\X)$ of $\X$ will be an analytic vector field induced on $\aS^2$ as follows:

First  we multiply $\hat{\X}$ by the factor $y_3^{2}$, and so the vector field $y_3^{2}\hat{\X}$ is defined in the whole $\aS^2$. Additionally, on $\aS^2\setminus \aS^1$ there are two symmetric copies of $\X$ and note that the behavior of $p(\X)$ around $\aS^1$ gives the behavior of $\X$ near the infinity. Then the {\it Poincar\'e disc} $\D^2$ is the projection of the closed northern hemisphere of $\aS^2$ on $y_3= 0$ under $(y_1, y_2, y_3) \longmapsto (y_1,y_2)$.

Since $\aS^2$ is a differentiable manifold, we can consider the six local charts $U_i = \{y \in \aS^2 : y_i > 0\}$, and $V_i = \{y \in \aS^2 : y_i < 0\}$ for $i = 1,2,3$ with the diffeomorphisms $F_i: U_i\longrightarrow \R^2$ and $G_i: V_i\longrightarrow \R^2,$ which are the inverses of the central projections from the planes tangent at the points $(1, 0, 0), (- 1, 0, 0), (0, 1, 0), (0, - 1, 0), (0, 0, 1)$ and $(0, 0, - 1)$ respectively. We set $z = (z_1, z_2)$ to be the value of $F_i(y)$ or $G_i(y)$ for any $i = 1,2,3.$  Hence, the expressions of the compactified vector field
$p(\X)$ of $\X$  are
\begin{equation*}
\begin{aligned}
z^{2}_{2}\Delta(z)\Bigg( Q\Big(\frac{1}{z_2},\frac{z_1}{z_2}\Big)-z_1P\Big(\frac{1}{z_2},\frac{z_1}{z_2}\Big),\,-z_2P\Big(\frac{1}{z_2}, \frac{z_1}{z_2}\Big) \Bigg) \; \qquad \mbox{in}   \qquad U_1,\\
z^{2}_{2}\Delta(z)\Bigg( P\Big(\frac{z_1}{z_2},\frac{1}{z_2}\Big)- z_1Q\Big(\frac{z_1}{z_2},\frac{1}{z_2}\Big),\,-z_2Q\Big(\frac{z_1}{z_2}, \frac{1}{z_2}\Big) \Bigg) \;  \qquad \mbox{in} \qquad U_2,\\
\Delta(z)\big(P(z_1,z_2),Q(z_1,z_2)\big) \; \qquad \mbox{in} \qquad U_3,
\end{aligned}
\end{equation*}
where $\Delta(z) = (z_1^2 + z_2^2 + 1)^{- \frac{1}{2}}$. The expressions of the vector field $p(\X)$ in the local chart $V_i$ is the same as in the chart $U_i$ multiplied by the  factor $-1$. In these coordinates $z_2 = 0$ denotes the points of $\aS^1$. Usually, we omit the factor $\Delta(z)$ by rescaling the vector field $p(\X),$ and so we obtain a polynomial vector field in each local chart. Also note that the infinity $\aS^1$ is invariant with respect to $p(\X)$.

Two polynomial vector fields $\X$ and $\Y$ on $\R^2$ are {\it topologically equivalent} if there exists a homeomorphism on $\aS^2$ preserving the infinity $\aS^1$ carrying orbits of the flow induced by $p(\X)$ into orbits of the flow induced by $p(\Y)$. Note that the homeomorphism should preserve or reverse simultaneously the sense of all orbits of the two compactified vector fields $p(\X)$ and $p(\Y)$.

\section{Appendix 2: Separatrix configuration}

In order to proceed with the topological classification of the global phase portraits of system \eqref{e2} we need to consider the definition of parallel flows. We use the definition given by Markus \cite{markus} and Neumann in \cite{Neumann}. Let $\phi$ be a  ${\mathcal C}^k$ local flow on the two dimensional manifold $\R^2$ or $\R^2\setminus\{0\}$. Here $k$ is either a positive integer, or $\infty$ (smooth), or $\omega$ (analytic). The flow $(M, \phi)$ is ${\mathcal C}^k$ {\it parallel} if it is ${\mathcal C}^k$-equivalent to one of the following ones:
\begin{itemize}
\item[\it {strip}:] $(\R^2, \phi)$ with the flow $\phi$ defined by $\dot{x} = 1, \dot{y} = 0$;

\item[\it annular:] $(\R^2\setminus \{0\}, \phi)$ with the flow $\phi$ defined (in polar coordinates) by $\dot{r} = 0, \dot{\theta} =1$;

\item[\it spiral:] $(\R^2\setminus \{0\}, \phi)$  with the flow $\phi$ defined by $\dot{r} = r, \dot{\theta} = 1$.
\end{itemize}

The {\it separatrices} of the vector field $p(\X)$ in the Poincar\'e disc $D$ are
\begin{itemize}
\item[(i)] all the orbits of $p(\X)$ which are in the boundary $\aS^1$ of the Poincar\'e disc (recall that $\aS^1$ is the infinity of $\R^2$);

\item[(ii)] all the finite singular points of $p(\X)$;

\item[(iii)] all the limit cycles of $p(\X)$; and

\item[(iv)] all the separatrices of the hyperbolic sectors of the finite and infinite singular points of $p(\X)$.
\end{itemize}

We denote by $\Sigma$ the union of all separatrices of the flow $(\D,\phi)$ defined by the compactified vector field $p(\X)$ in the Poincar\'e disc $\D$. Note that $\Sigma$ is a closed invariant subset of $\D$. Every open connected component of $\D\setminus \Sigma$, with the restricted flow, is called a {\it canonical region} of $\phi$.

For a proof of the following result see \cite{LiNic, Neumann}.

\begin{theorem}\label{basic}
Let $\phi$ be a ${\mathcal C}^k$ flow in the Poincar\'e disc with finitely many separatrices, and let $\Sigma$ be the union of all its separatrices. Then the flow restricted to every canonical region is ${\mathcal C}^k$ parallel.
\end{theorem}

The {\it separatrix configuration} $\Sigma_c$ of a flow $(\D,\phi)$ is the union of all the separatrices $\Sigma$ of the flow together with an orbit belonging to each canonical region. The separatrix configuration $\Sigma_c$  of the flow $(\D,\phi)$ is said to be {\it topologically equivalent} to the separatrix configuration ${\tilde \Sigma}_c$ of the flow $(\D,\tilde{\phi})$ if there exists a homeomorphism from $\Sigma_c$ to $\tilde \Sigma_c$ which transforms orbits of $\Sigma_c$ into orbits of $\tilde{\Sigma}_c$, and orbits of $\Sigma$  into orbits of $\tilde{\Sigma}.$

For a proof of the next Theorem see \cite{markus, Neumann, peixoto}.

\begin{theorem}\label{mnp}
Let $(\D,\phi)$ and $(\D,\tilde{\phi})$ be two compactified Poincar\'e flows with finitely many separatrices coming from two polynomial vector fields \eqref{e2}. Then they are topologically equivalent if and only if their separatrix configurations are topologically equivalent.
\end{theorem}

Note that from Theorem \ref{mnp} in order to classify the phase portraits in the Poincar\'e disc of a planar polynomial differential system having finitely many separatrices, it is enough to describe their separatrix configuration.

\end{document}